\documentclass[11pt]{amsart}
\usepackage{enumitem}
\usepackage{amsmath}
\usepackage{amsthm}
\usepackage{amssymb}
\usepackage{pdfpages}
\usepackage{graphicx}
\usepackage{todonotes}
\usepackage{hyperref}
\usepackage[capitalize]{cleveref}

\newcounter{alph}

\newtheorem{theo}[alph]{Theorem}

\newtheorem{coro}[alph]{Corollary}
\numberwithin{equation}{section}
\newtheorem{cor}[equation]{Corollary}
\newtheorem{lem}[equation]{Lemma}

\newtheorem{prop}[equation]{Proposition}
\newtheorem{thm}[equation]{Theorem}
\theoremstyle{definition}

\newtheorem{remo}[alph]{Remark}

\DeclareMathOperator{\ess}{ess}

\DeclareMathOperator{\Mat}{Mat}

\DeclareMathOperator{\Ric}{Ric}

\DeclareMathOperator{\Hess}{Hess}
\DeclareMathOperator{\Hom}{Hom}

\def\C{\mathbb C}

\def\N{\mathbb N}
\def\R{\mathbb R}

\def\ve{\varepsilon}
\def\vf{\varphi}

\definecolor{blue(ncs)}{rgb}{0.0, 0.53, 0.74}
\definecolor{cyan}{rgb}{0.0, 0.72, 0.92}
\definecolor{green}{rgb}{0.0, 0.5, 0.35}

\begin{document}

\title[Spectral stability of coverings]{On the spectral stability of finite coverings}
\author{Werner Ballmann}
\address
{WB: Max Planck Institute for Mathematics, Vivatsgasse 7, 53111 Bonn}
\email{hwbllmnn\@@mpim-bonn.mpg.de}
\author{Sugata Mondal}
\address{SM: Department of Mathematics and Statistics, University of Reading, UK}
\email{s.mondal@reading.ac.uk}
\author{Panagiotis Polymerakis}
\address{PP: Department of Mathematics, University of Thessaly, 3rd km Old National Road Lamia–Athens, 35100, Lamia, Greece}
\email{ppolymerakis\@@uth.gr}

\date{August 7, 2025}


\thanks{\emph{Acknowledgments.}
We are grateful to the Max Planck Institute for Mathematics
and the Hausdorff Center for Mathematics in Bonn for their support and hospitality. S. M. was partially supported by EPSRC grant APP16691.
We thank Will Hide for drawing our attention to the recent articles \cite{HideMoyNaud25} and \cite{MageePudervanHandel25}.}

\maketitle

\begin{abstract}
We prove the non-existence of new eigenvalues in $[0,\Lambda]$ for specific and random finite coverings of a complete and connected Riemannian manifold $M$ with Ricci curvature bounded from below, where $\Lambda$ is any positive number below the essential spectrum of $M$ and the spectrum of the universal cover of $M$, provided the representation theory of the fundamental group of $M$ satisfies certain conditions.
\end{abstract}

\setcounter{tocdepth}{1}
\tableofcontents

\section{Introduction}
The spectrum $\sigma$ of the Laplacian $\Delta$ of a complete and connected Riemannian manifold $M$ decomposes into a discrete part $\sigma_d$ and an essential part $\sigma_{\ess}$, the first one consisting of those $\lambda\in\sigma$ such that $\Delta-\lambda$ is a Fredholm operator.
We denote the bottom of $\sigma$ by $\lambda_0=\lambda_0(M)$ and that of $\sigma_{\ess}$ by $\lambda_{\ess}=\lambda_{\ess}(M)$.
Both, $\lambda_0$ and $\lambda_{\ess}$, do not change under finite Riemannian coverings.
By definition, the spectrum of $M$ below $\lambda_{\ess}$ consists of isolated eigenvalues of finite multiplicity. 

We are concerned with the relation of spectra below $\lambda_{\ess}$ under finite Riemannian coverings.
To that end, let $p\colon M'\to M$ be a finite Riemannian covering of complete and connected Riemannian manifolds.
Since Laplacians are compatible with pull-backs, $\sigma(M)\subseteq\sigma(M')$.
We aim to identify intervals $I\subseteq[0,\lambda_{\ess}(M))$ such that, including multiplicities, 
\[
I \cap \sigma(M) = I \cap \sigma(M').
\]
If this holds, then we say that $p$ respectively $M^\prime$ is \emph{$I$-stable}.
In general, one cannot expect a finite covering to be $I$-stable if $I$ is large.
In fact, \cite{BM} contains a number of results on \emph{$I$-instability}.
 
Over the last few years, Magee and coauthors developed a general approach that they used to show the existence of intervals $I$ and finite coverings $p\colon M'\to M$ that are $I$-stable when $M$ is a surface of a certain type; see below.  
In fact, in these situations, as $n$ - the number of sheets of the covering $p\colon M'\to M$ - tends to $\infty$, the proportion of $I$-stable $n$-sheeted coverings, among all $n$-sheeted coverings, approaches one.
The authors call this phenomenon \emph{uniform spectral gap for random finite coverings}.
Following \cite{BM}, we refer to this phenomenon as asymptotically almost sure $I$-stability of finite coverings.

While the very first uniform spectral gap result is due to Brooks-Makover \cite{BrooksMakover04}, prior to the work \cite{MageeNaud20}, examples of $I$-stable finite coverings, for any interval $I$ were rare. 
A result of Brooks \cite[Theorem 3]{Brooks86} shows the existence of such an interval for \emph{any} finite covering provided $\pi_1(M, x)$ has \emph{Property T}.  
Recall that fundamental groups of surfaces (this includes free groups) do not have Property T. 
In fact, a result of Randol \cite{Ra74} is quite opposite: for any \emph{closed} (i.e., \emph{compact and connected with empty boundary}) hyperbolic surface $S$ of genus $g\ge2$, any natural number $\ell$, and any $\ve>0$,
there is a finite Riemannian covering $p\colon S'\to S$ such that $S'$ has at least $\ell$ eigenvalues in $[0,\ve)$.
Compare also with \cite{BM}.

To state the above mentioned results more precisely, let $S$ be a complete hyperbolic surface of finite type and without boundary.
Then the universal cover of $S$ is the hyperbolic plane $H$ with $\lambda_0(H)=1/4$.
For a covering $p\colon S'\to S$, we denote the number of sheets of this covering by $|p|$.

1) In \cite{MageeNaud20} (see also \cite{MageeNaud21}) it is shown that for $S$ orientable, non-compact and convex cocompact with Hausdorff dimension of its limit set $\delta>1/2$ and for any $\sigma\in(3\delta/4,\delta)$, finite Riemannian coverings $p\colon S'\to S$ are asymptotically almost surely $[0,\sigma(1-\sigma)]$-stable as $|p| \to \infty$.
    
2) In \cite{MageeNaudPuder22}, the authors show that for $S$ closed and orientable, and for any $\ve>0$, finite Riemannian coverings $p\colon S'\to S$ are asymptotically almost surely $[0,3/16-\ve]$-stable as $|p|\to\infty$.
This result was recently extended in \cite{MageePudervanHandel25} and \cite{HideMoyNaud25} to $[0, \lambda_0(\tilde{S})-\ve]$-stability for metrics of variable negative curvature.
    
3) In \cite{HideMagee21}, the authors show that for $S$ orientable and non-compact of finite area, and for any $\ve>0$, finite Riemannian coverings $p\colon S'\to S$ are asymptotically almost surely $[0,1/4-\ve]$-stable as $|p|\to\infty$.
This result was recently extended in \cite{Moy25} to any complete non-compact surface of finite type with pinched negative curvature $-b^2 \le K \le -1$.

The setup of our main result in the introduction is a complete and connected Riemannian manifold $M$ with fundamental group $\Gamma$ and universal covering space $\tilde M$, endowed with the lifted Riemannian metric.
We also assume that
\begin{align}\label{lambda}
    \mu = \min\{\lambda_{\ess}(M),\lambda_0(\tilde M)\}>0.
\end{align}
Two cases of interest are the following.
1) If $M$ is closed, then $\sigma_{\ess}(M)=\emptyset$ so that $\lambda_{\ess}(M)=\infty$.
Furthermore, by a celebrated result of Brooks \cite[Theorem 1]{Brooks81}, $\lambda_0(\tilde M)>0$ if and only if $\Gamma$ is non-amenable.
2) If $M$ is geometrically finite (in the sense of Bowditch \cite{Bo}), then $\lambda_{\ess}(M)\ge\lambda_0(\tilde M)>0$; see \cite[Theorem B]{BP24}.
In \cref{subcov}, we discuss \eqref{lambda} in more detail.

To state our main result, we need a property from random matrix theory.
Let $\rho_\infty$ be the regular representation of $\Gamma$ on $\ell^2(\Gamma)$ and $\rho_n^0$ the irreducible representation of the symmetric group $S_n$ on $V_n^0=\{x\in\C^n\mid \sum x_i=0\}$.
We call elements $\vf\in\Hom(\Gamma,S_n)$ and also their compositions $\rho_n^0\circ\vf$ \emph{permutation representations} of $\Gamma$.
For $\vf\in\Hom(\Gamma,S_n)$, we consider the following inequality:
\begin{align}\label{almost}
    \bigg\|\sum_{\gamma\in\Gamma} a(\gamma)\otimes(\rho_n^0\circ\vf)(\gamma)\bigg\|_{\C^r\otimes V_n^0} \le \bigg\|\sum_{\gamma\in\Gamma} a(\gamma)\otimes\rho_\infty(\gamma) \bigg\|_{\C^r\otimes\ell^2(\Gamma)} + \ve,
\end{align}
where $a$ is a finitely supported map from $\Gamma$ to $\Mat_{r\times r}(\C)$, for some $r\ge1$, and $\ve>0$.
The norms are to be understood as operator norms. 
The inequality \eqref{almost} enters our discussion, since, following \cite{HideMagee21}, we parametrize (possibly non-connected) finite covers of $M$ of degree $n$ by elements $\vf\in\Hom(\Gamma,S_n)$.

\begin{theo}\label{maintheo}
	Assume that the Ricci curvature of $M$ is bounded from below and \eqref{lambda} holds.
	Fix $0<\Lambda<\mu$ and $0<\ve<1$. Then there are integers $r_i\ge1$ and finitely supported maps $a_i\colon\Gamma\to\Mat_{r_i\times r_i}(\C)$, $1\le i\le k$ (for some--large--$k\in\N$), such that if $\vf\in\Hom(\Gamma,S_n)$ satisfies \eqref{almost} with input $r=r_i$, $a=a_i$, and the given $\ve$, for each $1\le i\le k$, then $M_\vf$ is $[0,\Lambda]$-stable.
\end{theo}

As we mentioned in the beginning, $\lambda_0$ and $\lambda_{\ess}$ do not change under finite coverings, which implies the statement of \cref{maintheo} if $\Lambda<\lambda_0(M)$; in particular, if $\lambda_0(M)=\mu$.
On the other hand, if $\Lambda\ge\lambda_0(M)$, then $\lambda_0(M)$ is an eigenvalue of $M$ and of $M_\vf$ with multiplicity equal to the number of connected components of $M$ and $M_\vf$, respectively. Then $[0,\Lambda]$-stability of $M_\vf$ implies that $M_\vf$ is connected, $M$ being connected.

In our proof of \cref{maintheo}, we adapt the strategy of Hide and Magee \cite{HideMagee21} in their proof of the geometric part of their Theorem 1.1.
In particular, we construct a parametrix using an appropriate decomposition of $M$ into a compact interior and a neighborhood of infinity (in case $M$ is non-compact) and use a finite-dimensional approximation to bring in \eqref{almost}.
The decomposition is reflected by \eqref{lambda}, that is, by the two assumptions $\lambda_0(\tilde M)>0$ and $\lambda_{\ess}(M)>0$.

We consider now the following asymptotic property of $\Gamma$:
\begin{enumerate}[label=(PRP)]
    \item\label{pr}
For any given $r>0$, finitely supported map $a\colon\Gamma\to\Mat_{r\times r}(\C)$, and $\ve>0$, the inequality \eqref{almost}
holds asymptotically almost surely for $\vf\in\Hom(\Gamma,S_n)$ (as $n\to\infty$).
\end{enumerate}

\noindent
Here PR stands for permutation representation and the P on the right for probability.
\ref{pr} is fulfilled by surface groups including finitely generated free groups; see \cref{subsecaas} below.
Our main application of \cref{maintheo} is the following.

\begin{theo}\label{mainth}
Assume that the Ricci curvature of $M$ is bounded from below, \eqref{lambda} holds, and $\Gamma$ satisfies \ref{pr}.
Then, for any $0<\Lambda<\mu$, finite Riemannian coverings $p\colon M'\to M$ are asymptotically almost surely $[0,\Lambda]$-stable.
\end{theo}

\cref{mainth} is an immediate consequence of \cref{maintheo}.
It generalizes the corresponding results mentioned further up, since the involved fundamental groups satisfy \ref{pr}.
Indeed, we have the following consequence.

\begin{coro}\label{maincor}
Let $S$ be a surface of finite type with negative Euler characteristic, orientable if closed, endowed with a complete Riemannian metric with curvature bounded from below.
Assume that $\lambda_0(\tilde S)>0$.
Then $S$ satisfies \eqref{lambda} and, for any $0<\Lambda<\lambda_0(\tilde S)$, finite Riemannian coverings $p\colon S'\to S$ are asymptotically almost surely $[0,\Lambda]$-stable.
\end{coro}

Recall that $\lambda_0(\tilde S)>0$ if $S$ is closed (see \cite{Brooks81} respectively \cref{thm:Brooks} below) or, more generally, if $\lambda_0(S)<\lambda_{\ess}(S)$ (see \cite{Po} respectively \cref{thmpp}), or if the curvature of $S$ has a negative upper bound.

\begin{proof}[Proof of \cref{maincor}]
By \cref{amena}, $\lambda_{\ess}(S)\ge\lambda_0(\tilde S)$, and hence $\mu=\lambda_0(\tilde S)>0$.
If $S$ is non-compact, then the fundamental group of $S$ is a finitely generated free group and, hence, satisfies \ref{pr}.
If $S$ is orientable and closed, then the fundamental group of $S$ satisfies \ref{pr} as well.
\end{proof}

There is a second asymptotic property, which we consider now.
\begin{enumerate}[label=(PRS)]
    \item\label{prs}
There is a sequence of $\vf_i\in\Hom(\Gamma,S_{n_i})$ with $n_i\to\infty$,
such that, for any given $r>0$, finitely supported map $a\colon\Gamma\to\Mat_{r\times r}(\C)$, and $\ve>0$, there is an $i_0$, such that the inequality \eqref{almost} holds for all $\vf_i$ with $i\ge i_0$.
\end{enumerate}

As above, PR stands for repermutation presentation and the S on the right for sequence.
This property appears as P\textbf{Perm}F in \cite{Magee25}.
It holds for all fundamental groups of closed surfaces with Euler characteristic $\le-2$; see \cref{subsecaas}.

\begin{theo}\label{mainth2}
Assume that the Ricci curvature of $M$ is bounded from below, \eqref{lambda} holds, and $\Gamma$ satisfies \ref{prs}.
Then, for any $0<\Lambda<\mu$, $M$ admits a sequence of $[0,\Lambda]$-stable finite Riemannian coverings $p_i\colon M_i\to M$ of degree $n_i\to\infty$.
\end{theo}

For surfaces of finite type, we get the following consequence.

\begin{coro}\label{maincors}
Let $S$ be a surface of finite type with negative Euler characteristic, of Euler characteristic $\le-2$ if non-orientable and closed, endowed with a complete Riemannian metric with curvature bounded from below.
Assume that $\lambda_0(\tilde S)>0$.
Then $S$ satisfies \eqref{lambda} and, for any $0<\Lambda<\mu$, $S$ admits a sequence of $[0,\Lambda]$-stable finite Riemannian coverings $p_i\colon S_i\to S$ of degree $n_i\to\infty$.
\end{coro}

\begin{proof}
By \cref{amena}, $\lambda_{\ess}(S)\ge\lambda_0(\tilde S)$, and hence $\mu=\lambda_0(\tilde S)>0$.
If $S$ is non-compact, then the fundamental group of $S$ is a finitely generated free group and, hence, satisfies \ref{prs}.
If $S$ is closed, of Euler characteristic $\le-2$ if non-orientable, then the fundamental group of $S$ satisfies \ref{prs}.
\end{proof}

\begin{remo}
In Corollaries \ref{maincor} and \ref{maincors}, we discuss applications of our main results to surfaces.
There are also applications to higher dimensional manifolds, and to hyperbolic manifolds in particular.
Michael Magee pointed out to us the very interesting case of quasi-Fuchsian threefolds of type one with a reference to Rufus Bowen \cite{Bowen79}.
Another important class of examples are Schottky manifolds.
\end{remo}

By \cite[Lemma 2.1]{Magee25}, if $\Gamma$ satisfies \ref{prs}, then any subgroup of $\Gamma$ of finite index also satisfies \ref{prs}. Therefore, \cref{mainth2} has another immediate consequence.

\begin{coro}\label{mainco}
Assume that the Ricci curvature of $M$ is bounded from below, \eqref{lambda} holds, and $\Gamma$ satisfies \ref{prs}.
Then, for any $0<\Lambda<\mu$, $M$ admits a tower of $[0,\Lambda]$-stable finite Riemannian coverings.
\end{coro}

The setup in the text is more general than in the introduction.
Instead of universal coverings and fundamental groups, we consider infinite normal coverings $\hat p\colon\hat M\to M$ and finite   intermediate coverings $p\colon M'\to M$, where intermediate means that $\hat p=p\circ p'$ for a covering $p'\colon\hat M\to M'$.

\subsection{On properties \ref{pr} and \ref{prs}}\label{subsecaas}
Finding groups which satisfy \ref{pr} or \ref{prs} is an important question with a long history.
We state some results here and refer the interested reader to \S 1.4 of \cite{MageePudervanHandel25} for a more detailed account. 

It follows from the breakthrough result \cite{BordenaveCollins19} of Bordenave-Collins that finitely generated free groups satisfy \ref{pr}.
Observe here that the fundamental group of a non-compact orientable surface of finite topological type and negative Euler characteristic is a free group with finitely many (and at least two) generators. 
Very recently Magee-Puder-van Handel \cite{MageePudervanHandel25} have shown that the fundamental group of a closed orientable surface of genus at least two satisfies \ref{pr}.

In \cite[Theorem 1.1]{LouderMagee25}, Louder-Magee showed that limit groups satisfy \ref{prs}. It is noteworthy that fundamental groups of orientable closed surfaces \cite{B62}, as well as fundamental groups of non-orientable closed surfaces with Euler characteristic at most $-2$ (cf. \cite[pp. 414--415]{B67}) are limit groups.

\subsection{Structure of the paper}
We collect some preliminary results in \cref{prelim}. These include some basics of $\sigma(\Delta)$ on a Riemannian manifold, that of a finite cover, inequalities between $\lambda_0(\bar{M})$ and $\sigma_{\ess}(M)$, parametrisation of finite coverings and function spaces on them and two geometric inequalities. Although our approach follows the strategy in \cite{HideMagee21} closely, unlike in the hyperbolic plane, we do not have an explicit global expression for the resolvent operator of the Laplacian on $\hat{M}$.
Due to this, we heavily rely on the gradient estimate of Cheng and Yau \cite[Theorem 6]{ChengYau}.
It is worth pointing out that the results \cite{HideMagee21}, \cite{HideMoyNaud25} and \cite{Moy25} heavily depend on the asymptotic behaviour of the heat kernel at infinity. Our approach, thanks to the Cheng-Yau estimate, is completely free from such asymptotic behaviour.
We restate the Cheng-Yau estimate in the form we need it in \cref{CY}.

In \cref{secparex}, we start by discussing the decomposition of $M$ into a compact interior part and a neighborhood of infinity.
Using the characterization of the bottom of the essential spectrum as the limiting $\lambda_0(M\setminus K)$, where $K$ runs over the set of compact subsets of $M$, and a smoothing result of Cheeger and Colding \cite[Theorem 6.3]{CheegerColding},
we construct a parametrix for a neighborhood of infinity.
Both, the Cheng-Yau estimate and the Cheeger-Colding result, assume a lower bound on the Ricci curvature.
They are the reason for the corresponding assumption in our main theorem.
We restate the Cheeger-Colding result in the form we need it in \cref{CC}.

In \cref{secparin}, we commence our discussion of the interior parametrix. In the first step, we appropriately modify the resolvent on $\tilde M$. This part is significantly more complicated compared to \cite{HideMagee21} due to the non-explicit global expression of the resolvent operator on $\tilde M$. We develop the corresponding kernel estimates in \cref{seckerest} with the help of the Cheng-Yau gradient estimate \cref{CY}.
We then patch the ends and interior parametrices into a parametrix on $M'$ in \cref{secauto}.
Finally, we use estimates on the parametrix to prove an extended version of \cref{maintheo} in \cref{secrandom},
where a finite-dimensional approximation of the interior parametrix gives rise to the finitely supported maps $a_i$ in the statement of the theorem.

\section{Preliminaries}\label{prelim}
Let $M^m$ be a complete and connected Riemannian manifold.
Denote by $\Delta$ its Laplace-Beltrami operator, acting on the space $C^\infty(M)$ of smooth functions on $M$. 
Recall that $\Delta$ is {\em essentially self-adjoint} on $C^\infty_c(M)\subseteq L^2(M)$.
Its closure will also be denoted by $\Delta$.
The spectrum of the closure, depending on the context denoted by
\begin{align*}
    \sigma(M,\Delta) = \sigma(\Delta) = \sigma(M),
\end{align*}
can be decomposed into two sets,
\begin{align*}
    \sigma(M) = \sigma_d(M) \sqcup \sigma_{\ess}(M),
\end{align*}
the {\em discrete spectrum} and the {\em essential spectrum}.
Recall that $\sigma_d(M)$ consists of isolated eigenvalues of $\Delta$ of finite multiplicity
and that $\sigma_{\ess}(M)$ consists of those $\lambda\in\R$ for which $\Delta-\lambda$ is not a Fredholm operator.
By elliptic regularity theory, $\sigma(M)=\sigma_d(M)$ if $M$ is compact.
By the above characterization of the discrete spectrum, $\sigma(M)=\sigma_{\ess}(M)$ if $M$ is homogeneous and non-compact.

Denote by $\lambda_0(M)\le\lambda_{\ess}(M)$ the bottom of $\sigma(M)$ and $\sigma_{\ess}(M)$, respectively.
Recall that
\begin{align}\label{lambdaz}
    \lambda_0(M) = \inf_f \int_M|\nabla f\|^2/\int_M f^2 = \|\nabla f\|_{L^2}^2/|f\|_{L^2}^2,
\end{align}
where the infimum is taken over all non-zero $f\in C^\infty_c(M)$.
If $M$ is compact, then $\lambda_0(M)=0<\lambda_{\ess}(M)=\infty$.
Furthermore, $0$ is an eigenvalue of $\Delta$ of multiplicity one with constant functions as eigenfunctions.
In general, $\lambda_0(M)$ may be positive and may belong to $\sigma_{d}(M)$
or we may have $\lambda_0(M)=\lambda_{\ess}(M)$.

\subsection{Spectrum under finite Riemannian coverings}\label{subcov}
Let $p\colon M'\to M$ be a finite Riemannian covering.
Then we have
\begin{align}\label{bottomeq}
    \lambda_0(M')=\lambda_0(M) \quad\text{and}\quad \lambda_{\ess}(M')=\lambda_{\ess}(M).
\end{align}
The first equality follows easily from \eqref{lambdaz}, the second from \cref{lambdae}, observing that \eqref{lambdaz} also holds for Dirichlet spectra.

Now we recall some results on the bottom of spectra, starting with a classical result of Brooks \cite{Brooks81}.

\begin{thm}\label{thm:Brooks}
If $M$ is a closed and connected Riemannian manifold, then $\lambda_0(\tilde{M})>0$ if and only if $\pi_1(M)$ is non-amenable.
\end{thm}

After a previous extension of this result in \cite{BMP20}, where a lower bound on the Ricci curvature is assumed, the most general version to date (actually for scalar Schr\"odinger operators) was proved by the last named author as follows \cite{Po}.

\begin{thm}\label{thmpp}
If $q\colon \hat M\to M$ is a Riemannian covering such that $\lambda_0(M)<\lambda_{\ess}(M)$,
then $\lambda_0(\hat M)>\lambda_0(M)$ if and only if $q$ is non-amenable.
\end{thm}

Here the covering is called amenable if the right action of the fundamental group of $M$ on the fibers of $q$ is amenable.

Next we discuss the essential spectrum and its bottom.
If $M$ is closed and connected, then the essential spectrum of $M$ is empty, that is, $\lambda_{\ess}(M)=\infty$.
It is the special case of the following classical result, different versions of which have been proved at different times by different authors.

\begin{thm}\label{lambdae}
For any complete Riemannian manifold $M$, \[\lambda_{\ess}(M)=\sup\lambda_0(M\setminus K^\circ),\]
where the supremum is taken over all compact domains $K$ in $M$ with smooth boundary and $\lambda_0$ on the right refers to Dirichlet spectra.
\end{thm}

Since amenable Riemannian coverings preserve the bottom of Dirichlet spectra, we have the following immediate consequence, which may be found for instance in \cite[Theorem 7.10]{Po2}.

\begin{thm}\label{amena}
Let $M$ be a complete Riemannian manifold and suppose that there exists a compact $K\subset M$ with smooth boundary such that each connected component of $M\setminus K$ has amenable fundamental group.
Then $\lambda_{\ess}(M)\ge\lambda_0(\tilde{M})$.
\end{thm}

In particular, any hyperbolic surface $S$ of finite topological type satisfies the inequality $\lambda_{\ess}(S) \ge \lambda_0(\tilde{S})=1/4$.
A more general class of manifolds that satisfies this latter kind of inequality is given by the following special case of \cite[Theorem B]{BP24}.

\begin{thm}\label{geofin}
If $M$ is non-compact and geometrically finite (in the sense of Bowditch \cite{Bo}), then $\lambda_{\ess}(M)\ge\lambda_0(\tilde{M})$ with equality if $\tilde M$ is symmetric.
\end{thm}

\subsection{Finite coverings and associated function spaces}
We adapt the terminology of \cite{HideMagee21} to our setup.
Let $\hat p\colon\hat M\to M$ be an infinite normal covering of connected manifolds.
Let $p\colon M'\to M$ be a finite \emph{intermediate} covering; that is, $p$ factors through a covering $p'\colon\hat M\to M'$.
Let $x\in M$, $x'\in M'$ and $\hat x\in\hat M$ be points with $p'(\hat x)=x'$ and $p(x')=x$.
Then the image of $\pi_1(\hat M,\hat x)$ in $\pi_1(M,x)$ is a normal subgroup and consists of homotopy classes of loops which lift to closed loops in $\hat M$, regardless of the initial points of the lifts. 
In other words, lifting of loops defines a right action of the quotient group $\pi_1(\hat M,\hat x)/\pi_1(M,x)$ on the fiber of $\hat p$ over $x$.
Now $\Gamma\cong\pi_1(\hat M,\hat x)/\pi_1(M,x)$ with the latter action turned into a left action.

Since the lift of loops at $x$ to $M'$ consist of the images of their lifts to $\hat M$,
we obtain a right action of $\Gamma$ on the fiber $F_x$ of $p$ over $x$,
and we may think of $M'$ as the quotient space of $\hat M\times F_x$ via the diagonal action of $\Gamma$,
where we turn the right action on $F_x$ into a left action.

If $p$ has $n$ sheets, then we may identify $F_x=\{1,\dots,n\}$ and get a homomorphism $\vf\colon\Gamma\to S_n$, the symmetric group on $n$ letters.
Thus
\begin{align}\label{mvf}
	M' \sim M_\vf = \Gamma\backslash(\hat M\times\{1,\dots,n\}),
\end{align}
where the action of $\Gamma$ is given by $g(y,i) = (gy,\vf(g)i)$ and $p(\Gamma(y,i))=\hat p(y)=\Gamma y$.
Conversely, any $\vf\in\Hom(\Gamma,S_n)$ defines in this way an $n$-sheeted, possibly not connected, intermediate covering of $M$.
The induced cover $M_\vf$ is connected if the action of $\vf(\Gamma) \subset S_n$ is transitive on $\{1,\dots,n\}$.

Clearly, for $M'=M_\vf$ as in \eqref{mvf}, $C^\infty(M_\vf)$ consists of $\Gamma$-invariant smooth functions on $\hat M\times\{1,\dots,n\}$.
By setting
\begin{align}
    f_i(y) = f(y,i),
\end{align}
they may be viewed as $\Gamma$-equivariant smooth functions $\hat M\to\R^n$, where $\Gamma$ acts on $\R^n$ by
\begin{align}
    (\rho_n\circ\vf)(g)(x_1,\dots,x_n) = (x_{\vf(g)^{-1}(1)},\dots,x_{\vf(g)^{-1}(n)}).
\end{align}
Such a function $f$ is the pull back of a function on $M$ if $f_1(y)=\cdots=f_n(y)$.
We denote the space of pull backs of smooth functions by $C^\infty_{\textrm{old}}(M_\vf)$ and identify it with $C^\infty(M)$.

To compute $L^2$-norms, we choose a fundamental domain $F$ for $\Gamma$ and integrate over $F$.
Clearly, the $L^2$-orthogonal complement of $C^\infty_{\textrm{old}}(M_\vf)$ in $C^\infty(M_\vf)$ consists of the space of $f$ such that $f_1(y)+\dots+ f_n(y)=0$, which we denote by $C^\infty_{\textrm{new}}(M_\vf)$.
In other words, $C^\infty_{\textrm{new}}(M_\vf)$ consists of $\Gamma$-equivariant smooth functions from $\hat M$ to $V_n^0=\{x\in\R^n\mid \sum x_i=0\}$; that is,
\begin{align}
    C^\infty_{\textrm{new}}(M_\vf) = \{f\in C^\infty(\hat M,V_n^0) \mid \forall g\in\Gamma\colon f\circ g=(\rho_n^0\circ\vf)(g)\circ f\},
\end{align}where $\rho_n^0$ is the $(n-1)$-dimensional irreducible representation of $S_n$ on $V^n_0$.

We use similar terminology for $L^2$ and other spaces.
It is essential that the orthogonal decomposition \[L^2(M_\vf)=L^2_{\textrm{old}}(M_\vf)\oplus L^2_{\textrm{new}}(M_\vf)\] is invariant under the Laplacian.
Thus $I$-stability of $p\colon M_\vf\to M$ means that the spectrum of the Laplacian on $L^2_{\textrm{new}}(M_\vf)$ does not intersect $I$.

\subsection{Geometric estimates}
We will exploit the following result by Cheeger and Colding \cite[Theorem 6.33]{CheegerColding} (see also \cite[Theorem 4.13]{CheegerJiangNaber}).

\begin{thm}\label{CC}
Let $M^m$ be a complete Riemannian manifold with $\Ric_M \geq - (m-1)b^2$ for some $b > 0$. Then for any $R > b/2$ and $x \in M$ there exists a smooth $\chi \colon M \to [0,1]$ supported in $B(x,2R)$, with $\chi = 1$ in $B(x,R)$, such that
\[
R^2 \| \nabla \chi \|_\infty^2 + R^2 \|\Delta \chi\|_\infty \leq c(m),
\]
where $c(m)$ is a constant depending only on $m$.
\end{thm}

We will also need the following special case of the gradient estimate of Cheng and Yau \cite[Theorem 6]{ChengYau}.

\begin{thm}\label{CY}
Let $M^m$ be a complete Riemannian manifold with $\Ric_M \geq - (m-1)b^2$ for some $b > 0$. Let $f \in C^\infty(B(x,R))$ be a positive function satisfying $\Delta f \geq 0$ and
\[
\| \nabla (\Delta f) \| \leq C \| \Hess f \| + C \| \nabla f \|
\]
for some $R \geq \rho > 0$ and $C>0$. Then there exists $c>0$ depending only on $m$, $b$, $\rho$ and $C$ such that $\| \nabla f \| \leq c f$ in $B(x,R/2)$.
\end{thm}

\subsection{Conventions}
Unless specified otherwise, Riemannian manifolds $M^m$ are assumed to be complete and connected with Ricci curvature bounded from below, $\Ric_M \geq -(m-1)b^2$ for some $b > 0$.
Moreover, $\hat p\colon\hat M\to M$ denotes an infinite normal Riemannian covering with group $\Gamma$ of deck transformations and $p\colon M'\to M$ a finite intermediate Riemannian covering with $|p|$ sheets.

\section{End parametrix}\label{secparex}

Throughout this section, let $M^m$ be a non-compact (complete and connected) Riemannian manifold with $\Ric_M \geq -(m-1)b^2$ and $\lambda_{\ess}(M) > 0$, and fix $0 < \Lambda < \lambda_{\ess} (M)$ and $\delta > 0$.
By \cref{lambdae}, there exists a compact domain $K\subseteq M$ with smooth boundary such that
\[
\lambda_0(M \setminus K^\circ) \geq \Lambda + \frac{\lambda_{\ess}(M) - \Lambda}{2},
\]
where $\lambda_0(M \setminus K^\circ)$ stands for the bottom of the Dirichlet spectrum of $M \setminus K^\circ$.
Choose $x \in K$ and $R_0 > b/2$ such that $K \subset B(x,R_0)$. It follows from \cref{CC} that there exists $R>R_0$ and smooth $\psi \colon M \to [0,1]$ supported in $B(x,2R)$, with $\psi = 1$ in $B(x,R)$ such that
\[\| \nabla \psi \|_\infty,\| \Delta \psi \|_\infty < \alpha \delta, \text{ where } \alpha = \frac{\lambda_{\ess}(M) - \Lambda}{3(\lambda_{\ess}(M) + \Lambda + 2)}.
\]
Then $\chi_K^+ = 1 - \psi$ is smooth, vanishes in a neighborhood of $K$ and satisfies $\| \nabla \chi_K^+ \|_\infty,\| \Delta \chi_K^+ \|_\infty < \alpha \delta$. Choose also a smooth $\chi_K^- \colon M \to [0,1]$ such that $\chi_K^- = 1$ in $M \setminus B(x,3R)$ and $\chi_K^- = 0$ in $B(x,2R)$. It is clear that $\chi_K^+ \chi_K^- = \chi_K^-$ and $\chi_K^-$ vanishes in a neighborhood of $K$.

Given a finite-sheeted covering $p_\vf \colon M_\vf \to M$, consider the lifted functions $\chi_{K,\vf}^{\pm} = \chi_K^{\pm} \circ p_\vf$, and observe that $\| \nabla \chi^+_{K,\vf} \|_{\infty} , \| \Delta \chi^+_{K,\vf}  \|_{\infty} < \alpha \delta$. Denote by 
\[
\Delta_D \colon \mathcal{D}(\Delta_D) \subset L^2(M_\vf \setminus p^{-1}_\vf(K^\circ)) \to L^2(M_\vf \setminus p^{-1}_\vf(K^\circ))
\]
the Dirichlet Laplacian. The fact that $p_\vf$ is finite-sheeted implies that $\lambda_0(M_\vf \setminus p_\vf^{-1}(K^\circ)) = \lambda_0(M \setminus K^\circ)$. Therefore, for any $0 \leq \lambda \leq \Lambda$, the resolvent 
\[
R_{D}(\lambda) = (\Delta_D - \lambda)^{-1} \colon  L^2(M_\vf \setminus p^{-1}(K^\circ)) \to \mathcal{D}(\Delta_D)
\]
exists and its norm as an operator to $L^2(M_\vf \setminus p^{-1}_\vf(K^\circ))$ satisfies
\[
\| R_{D}(\lambda)  \|_{L^2} = \frac{1}{d(\lambda, \sigma(\Delta_D))} \leq \frac{2}{\lambda_{\ess}(M) - \Lambda}.
\]
It follows from
\[
\Delta_D R_D(\lambda) f = f + \lambda R_D(\lambda) f 
\]
that
\[
\| \Delta_D R_D(\lambda) f \|_{L^2} \leq \big( 1 + \frac{2 \Lambda}{\lambda_{\ess}(M) - \Lambda} \big) \| f \|_{L^2}
\]
for any $f \in L^2(M_\vf \setminus p_\vf^{-1}(K^\circ))$. Therefore, regarding $\mathcal{D}(\Delta_D)$ endowed with the graph norm
\[
\| f \|_{\Delta_D} = \| f \|_{L^2} + \| \Delta_D f \|_{L^2},
\]
the resolvent is bounded with 
\begin{equation}\label{res norm}
\| R_D(\lambda) \|_{\Delta_D} \leq \frac{\lambda_{\ess}(M) + \Lambda + 2}{\lambda_{\ess}(M) - \Lambda} = \frac{1}{3 \alpha}. 
\end{equation}

Bearing in mind that $\chi^+_{K,\vf}$ has bounded gradient and Laplacian, and vanishes in a neighborhood of $p^{-1}_\vf(K)$, it is easily checked that 
\begin{equation}\label{Dir Lap}
	\Delta(\chi^+_{K,\vf} f) = \Delta_D(\chi^+_{K,\vf} f)
\end{equation}
for any $f \in \mathcal{D}(\Delta_D)$, and the operator $\chi^+_{K,\vf} \colon \mathcal{D}(\Delta_D) \to \mathcal{D}(\Delta)$ is bounded, where $\mathcal{D}(\Delta)$ stands for the domain of the Laplacian on $M_\vf$ and both spaces are endowed with the corresponding graph norms. Combining the above, we deduce that the operator 
\[
\mathbb{M}_\vf^K(\lambda) = \chi^+_{K,\vf} R_D(\lambda) \chi^-_{K,\vf} \colon L^2(M_\vf) \to \mathcal{D}(\Delta)
\]
is bounded for any $0 \leq \lambda \leq \Lambda$.

It is straightforward to verify that $R_{D}(\lambda) (f \circ p_{\vf}) = (R_{D,K}(\lambda) f) \circ p_{\vf}$ for any $f \in L^2(M \setminus K^\circ)$, where $R_{D,K}(\lambda)$ stands for the resolvent of the Dirichlet Laplacian on $M \setminus K^\circ$ at $\lambda$. Using this, together with the fact that $\chi_K^{\pm}$ are lifted functions from $M$, it is easily checked that
\[
\langle \mathbb{M}_\vf^K(\lambda) f , g \rangle_{L^2} = 0
\]
for any $f \in L^2_{\textrm{new}}(M_\vf)$ and $g \in L^2_\textrm{old}(M_\vf)$. Since $L^2(M_\vf)$ is the orthogonal sum of $L^2_{\textrm{new}}(M_\vf)$ and $L^2_{\textrm{old}}(M_\vf)$, this means that
\[
\mathbb{M}_\vf^K (L^2_{\textrm{new}}(M_\vf)) \subset L^2_{\textrm{new}}(M_\vf) \cap \mathcal{D}(\Delta).
\]

Consider now the operator
\[
\mathbb{L}_\vf^K(\lambda) = [\Delta_D , \chi^+_{K,\vf}] R_D(\lambda) \chi^-_{K,\vf} \colon L^2(M_\vf) \to L^2(M_\vf)
\]
with $0 \leq \lambda \leq \Lambda$. Bearing in mind (\ref{Dir Lap}), we compute
\begin{eqnarray}
	(\Delta - \lambda) \mathbb{M}_\vf^K (\lambda) f - \mathbb{L}_\vf^K(\lambda)f &=& (\Delta_D - \lambda) \mathbb{M}_\vf^K (\lambda) f - \mathbb{L}_\vf^K(\lambda)f \nonumber\\
	& =& \chi^+_{K,\vf} \chi^{-}_{K,\vf} f = \chi^{-}_{K,\vf} f\nonumber
\end{eqnarray}
for any $f \in L^2(M_\vf)$, where we used that $\chi_K^+ \chi_K^- = \chi_K^-$. This yields that
\begin{equation}\label{end relation}
	(\Delta - \lambda) \mathbb{M}_\vf^K(\lambda) = \chi^-_{K,\vf} + \mathbb{L}_\vf^K(\lambda)
\end{equation}
for any $0 \leq \lambda \leq \Lambda$. Since $\mathbb{M}_\vf$ maps $L^2_{\textrm{new}}(M_\vf)$ to itself, and $\Delta$ maps $L^2_{\textrm{new}}(M_\vf) \cap \mathcal{D}(\Delta)$ to $L^2_{\textrm{new}}(M_\vf)$, we deduce that
\[
\mathbb{L}_\vf^K(\lambda)(L^2_{\textrm{new}}(M_\vf)) \subset L^2_{\textrm{new}}(M_\vf).
\]

\begin{lem}\label{estl}
The operator $\mathbb{L}_\vf^K(\lambda)$ is bounded with norm
\[
\| \mathbb{L}_\vf^K(\lambda) \|_{L^2} \leq \delta
\]
for any $0 \leq \lambda \leq \Lambda$.
\end{lem}

\begin{proof}
To estimate the norm of $[\Delta_D , \chi^+_{K,\vf}] \colon \mathcal{D}(\Delta_D) \to L^2(M_\vf)$, we calculate
\[
[\Delta_D , \chi^+_{K,\vf}] f = (\Delta \chi^+_{K,\vf}) f - 2 \langle \nabla \chi^+_{K,\vf} , \nabla f \rangle,
\]
which implies that
\[
\| [\Delta_D , \chi^+_{K,\vf}] f \|_{L^2} \leq \| \Delta \chi^+_{K,\vf} \|_{\infty} \|f\|_{L^2} + 2 \| \nabla \chi^+_{K,\vf} \|_{\infty} \| \nabla f \|_{L^2}
\]
for any $f \in \mathcal{D}(\Delta_D)$.
Keeping in mind that  $\| \nabla \chi_{K,\vf}^+ \|_{\infty} , \| \Delta \chi_{K,\vf}^+ \|_{\infty} < \alpha \delta$, and
\[
\| \nabla f \|_{L^2} = \langle \Delta_D f , f \rangle_{L^2}^{1/2} \leq \| \Delta_D f \|_{L^2}^{1/2} \| f \|_{L^2}^{1/2} \leq \| f \|_{\Delta_D},
\]
we derive that
\[
	\| [\Delta_D , \chi^+_{K,\vf}] \||_{L^2} \leq \| \Delta \chi^+_{K,\vf} \|_{\infty} + 2 \| \nabla \chi^+_{K,\vf} \|_{\infty} < 3 \alpha \delta.
\]
Using this, (\ref{res norm}) and the fact that $0 \leq  \chi_K^-  \leq 1$, we conclude that
\[
\| \mathbb{L}_{\vf}^K(\lambda) g \|_{L^2} \leq 3 \alpha \delta \|  R_D(\lambda) (\chi^-_{K,\vf} g)  \|_{\Delta_D}  \leq \delta \| \chi^-_{K,\vf} g \|_{L^2} \leq  \delta \|  g \|_{L^2}
\]
for any $g \in L^2(M_\vf)$ and $ 0 \leq \lambda \leq \Lambda$.
\end{proof}

Recalling that $0 < \Lambda < \lambda_{\ess}(M)$, $\delta > 0$, and the finite-sheeted covering $p_\vf$ considered throughout this section are arbitrary, we obtain the following:

\begin{prop}\label{estl2}
For any $0 < \Lambda < \lambda_{\ess}(M)$ and $\delta > 0$, we can choose $K$ and $\chi^{\pm}_K$ as above, such that
\[
\| \mathbb{L}_\vf^K(\lambda) \|_{L^2} \leq \delta
\]
for any $n \in \mathbb{N}$, $\vf \in \Hom(\Gamma,S_n)$ and $0 \leq \lambda \leq \Lambda$.
\end{prop}

\section{Parametrix on $\hat M$}\label{secparin}

For any $\lambda < \lambda_0(\hat{M})$, the resolvent $R(\lambda) = (\Delta - \lambda)^{-1} \colon L^2(\hat{M}) \to L^2(\hat{M})$ is the integral operator with kernel
\begin{equation}\label{Green}
G_\lambda(x,y) = \int_0^{\infty} e^{\lambda t} p_t(x,y) dt,
\end{equation}
where $p_t(x,y)$ stands for the heat kernel of $\hat{M}$ (cf. \cite[Corollary 2.3]{Su87}.
Recall that $G_\lambda \in C^\infty(\hat{M} \times \hat{M} \setminus \{(x,x) \colon x \in \hat{M}\})$ and $G_\lambda(\cdot,y)$ is a $\lambda$-eigenfunction on $\hat{M} \setminus \{y\}$ for any $y \in \hat{M}$. Indeed, from the definition of the resolvent we get \[ (\Delta_x - \lambda) ~ G_\lambda (., y) = \delta_y \] where $\delta_y$ is the Dirac mass at $y$. Therefore, for $x \in \hat{M} \setminus y$, we have
\[
\Delta_x ~ G_\lambda(., y) = \lambda \cdot G_\lambda(., y).\]
It is easy to see that $G_\lambda(., y)$ satisfies the condition in \cref{CY} away from $y$.

Choose one point of the fiber of $\hat{p} \colon \hat M \to M$ over any point of $M$ and denote the set of these points by $W$.
Then for any $y \in \hat{M}$ there exist unique $x \in W$ and $\gamma \in \Gamma$ such that $y = \gamma x$.
It follows from \cref{CC} that for any $T \geq \max\{1,b/2\}$ and $y \in W$ there exists smooth $\chi_{T,y} \colon \hat{M} \to [0,1]$ supported in $B(y,2T)$ with $\chi_{T,y} = 1$ in $B(y,T)$ such that
\begin{equation}\label{bound}
\| \nabla \chi_{T,y} \|_\infty , \| \Delta \chi_{T,y} \|_\infty \leq \frac{c(m)}{T} ,
\end{equation}
where $c(m)$ is a constant depending only on $m$.
Given any $y \in \hat{M}$, define $\chi_{T,y} = \chi_{T,x} \circ \gamma^{-1}$, where $x \in W$ and $\gamma \in \Gamma$ are the unique elements such that $y = \gamma x$. Then (\ref{bound}) holds for any $y \in \hat{M}$ and $\chi_{T,\gamma y}(\gamma x) = \chi_{T,y}(x)$ for any $x,y \in \hat{M}$ and $\gamma \in \Gamma$.

Consider now the function $R_{T}(\lambda,\cdot,\cdot) \colon \hat{M} \times \hat{M} \smallsetminus \{(x,x) \colon x \in \hat{M}\} \to \mathbb{R}$ defined by
\[
	R_{T}(\lambda,x,y) = \chi_{T, y}(x) G_\lambda(x,y),
\]
for any $x,y \in \hat{M}$ with $x \neq y$, and the function $L_{T}(\lambda,\cdot,\cdot) \colon \hat{M} \times \hat{M} \to \mathbb{R}$ given by
\[
L_{T}(\lambda,x,y) = \Delta_x \chi_{T,y}(x) G_\lambda(x,y) -2 \langle \nabla_x \chi_{T,y}(x) , \nabla_x G_\lambda(x,y)  \rangle
\]
for any $x,y \in \hat{M}$.
Observe that
\begin{equation}\label{Gamma equiv}
	R_{T}(\lambda,\gamma x, \gamma y) = R_{T}(\lambda, x,  y) \text{, } 	L_{T}(\lambda,\gamma x, \gamma y) = L_{T}(\lambda, x,  y) 
\end{equation}
for any $x , y \in \hat{M}$ and $\gamma \in \Gamma$.
Since
\[
R_{T}(\lambda,x,y) \leq  G_{\lambda}(x,y)
\]
for any $x \neq y$, we obtain that $R_{T}(\lambda,x,y)$ is the kernel of a bounded integral operator $R_{T}(\lambda) \colon L^2(\hat{M}) \to L^2(\hat{M})$ with 
\begin{equation}\label{R norm}
    \| R_{T}(\lambda) \| \leq \| R(\lambda) \|.
\end{equation}
Regarding $L_{T}(\lambda,x,y)$, we have the following:

\begin{prop}\label{L estimate}
There exists $C>0$ independent from $\lambda$ and $T$, such that
\[
|L_{T}(\lambda,x,y)| \leq \frac{C}{T} G_{\lambda}(x,y)
\]
for any $x,y \in \hat{M}$, $0 \leq \lambda < \lambda_0(\hat{M})$ and $T \geq \max\{1,b/2\}$. Moreover, $L_{T}(\lambda,x,y)$ is the kernel of a bounded operator $\mathbb{L}_{T}(\lambda) \colon L^2(\hat{M}) \to L^2(\hat{M})$ with
\[
\| \mathbb{L}_{T}(\lambda) \| \leq \frac{C}{T(\lambda_0(\hat{M}) - \lambda)}
\]
for any $0 \leq \lambda <\lambda_0(\hat{M})$ and $T \geq \max\{1,b/2\}$.
\end{prop}

\begin{proof}
Since $G_\lambda(\cdot , y)$ is positive and a $\lambda$-eigenfunction on $\hat{M} \setminus \{y\}$, it follows from \cref{CY} that there exists $c > 0$ depending only on $m$, $b$ and $\lambda_0(\hat{M})$ such that
\[
\| \nabla_x G_\lambda(x,y) \| \leq c  G_\lambda(x,y)
\]
for any $\lambda \in [0, \lambda_0(\hat{M}))$ and $x,y \in \hat{M}$ with $d(x,y) \geq 1/2$. This, together with (\ref{bound}) and the fact that $\chi_{T,y}$ is constant in the ball of radius $T > 1/2$ about $y$, gives the estimate
\[
|\langle \nabla_x \chi_{T,y}(x) , \nabla_x G_\lambda(x,y)  \rangle| \leq \frac{c(m) c}{T} G_\lambda(x,y)
\]
for any $x,y \in \hat{M}$. Using again (\ref{bound}), we derive that there exists $C > 0$ depending only on $m, b$ and $\lambda_0(\hat{M})$ such that
\[
|L_{T}(\lambda,x,y)| \leq \frac{C}{T} G_\lambda(x,y)
\]
for any $x,y \in \hat{M}$. This gives the estimate
\[
\big|\int_{\hat{M}}L_{T}(\lambda,x,y)f(y) dy \big| \leq \int_{\hat{M}} |L_{T}(\lambda,x,y) f(y)| dy \leq \frac{C}{T} R(\lambda)|f|(x)
\]
for any $f \in L^2(\hat{M})$ and almost any $x \in \hat{M}$.
This yields that the kernel $L_{T}(\lambda,x,y)$ defines an integral operator $\mathbb{L}_{T}(\lambda)$ on $L^2(\hat{M})$ satisfying
\[
\| \mathbb{L}_{T}(\lambda)f \|_{L^2} \leq \frac{C}{T} \| R(\lambda)|f| \|_{L^2} \leq \frac{C}{T} \| R(\lambda) \| \| f \|_{L^2}
\]
for any $f \in L^2(\hat{M})$. The proof is completed from the fact that $\| R(\lambda) \| = (\lambda_0(\hat{M}) - \lambda)^{-1}$.
\end{proof}

\begin{prop}\label{R est}
For any $0 \leq \lambda < \lambda_0(\hat{M})$ and $T \geq \max\{1 , b/2\}$, we have that $R_{T}(\lambda) f \in \mathcal{D}(\Delta)$ and
\[
(\Delta - \lambda) R_{T}(\lambda) f = f + \mathbb{L}_{T}(\lambda)f
\]
for any $f \in L^2(\hat{M})$.
\end{prop}

\begin{proof}
For any $f \in L^2(\hat{M})$ and $g \in C^\infty_c(\hat{M})$, we compute
\[
\langle R_{T}(\lambda) - R(\lambda))f , (\Delta -\lambda) g \rangle_{L^2} = \int_{\hat{M}} \int_{\hat{M}}K_y(x) f(y) (\Delta - \lambda) g(x) dy dx,
\]
where $K_y(x) = (\chi_{T,y}(x) - 1) G_{\lambda}(x,y)$.
Since
\[
\int_{\hat{M}} \int_{\hat{M}} |K(x,y) f(y) (\Delta - \lambda) g(x)| dy dx \leq  \langle R(\lambda) |f| , |(\Delta - \lambda)g| \rangle_{L^2} < \infty,
\]
we obtain from Tonelli's and Fubini's theorems that
\[
\langle R_{T}(\lambda) - R(\lambda))f , (\Delta -\lambda) g \rangle_{L^2} = \int_{\hat{M}} \int_{\hat{M}} K_y(x) f(y) (\Delta - \lambda) g(x) dx dy.
\]
Since $\chi_{T,y} = 1$ near $y$ and $G_{\lambda}(\cdot,y)$ is a $\lambda$-eigenfunction away from $y$, and thus, $K_y \in C^\infty(\hat{M})$, using the divergence formula, we derive that
\[
\langle R_{T}(\lambda) - R(\lambda))f , (\Delta -\lambda) g \rangle_{L^2} = \int_{\hat{M}} \int_{\hat{M}} L_{T}(\lambda,x,y) f(y) g(x) dx dy.
\]
We know from Proposition \ref{L estimate} that there exists $C>0$ such that
\[
|L_{T}(\lambda,x,y)| \leq \frac{C}{T} G_{\lambda}(x,y)
\]
for any $x,y \in \hat{M}$. This gives that
\begin{eqnarray}
\int_{\hat{M}} \int_{\hat{M}} |L_{T}(\lambda,x,y) f(y) g(x)| dx dy &\leq& \frac{C}{T} \int_{\hat{M}} \int_{\hat{M}} G_{\lambda}(x,y) |f(y) g(x)| dx dy \nonumber \\
&=& \frac{C}{T} \langle |f| , R(\lambda) |g| \rangle_{L^2} < \infty. \nonumber
\end{eqnarray}
We deduce from Tonelli's and Fubini's theorems that
\begin{eqnarray}
\langle R_{T}(\lambda) - R(\lambda))f , (\Delta -\lambda) g \rangle_{L^2} &=& \int_{\hat{M}} \int_{\hat{M}} L_{T}(\lambda,x,y) f(y) g(x) dy dx \nonumber\\
&=& \langle \mathbb{L}_{T}(\lambda) f , g \rangle_{L^2}. \nonumber
\end{eqnarray}
Since this holds for any $g \in C^\infty_c(\hat{M})$, while $\Delta \colon C^\infty_c(\hat{M}) \subset L^2(\hat{M}) \to L^2(\hat{M})$ is essentially self-adjoint, we conclude that $(R_{T}(\lambda) - R(\lambda))f \in \mathcal{D}(\Delta)$ and 
\[
(\Delta - \lambda)(R_{T}(\lambda) - R(\lambda))f = \mathbb{L}_{T}(\lambda) f
\]
for any $f \in L^2(\hat{M})$.
\end{proof}

\section{Estimates for the kernels}\label{seckerest}

\begin{lem}\label{Green estimate}
For any compact $K \subset \hat{M}$ and $0 < \Lambda < \lambda_0(\hat{M})$, there exists $C(K,\Lambda) > 0$ such that 
\[
G_{\lambda_2}(x,y) - G_{\lambda_1}(x,y) \leq C(K,\Lambda)(\lambda_2 - \lambda_1)
\]
for any $0 \leq \lambda_1 < \lambda_2 \leq \Lambda$ and $x, y \in K$ with $d(x,y) \geq 1/2$.
\end{lem}

\begin{proof}
We obtain from \ref{Green} that
\[
G_{\lambda_2}(x,y) - G_{\lambda_1}(x,y) = \int_{0}^{\infty} (e^{\lambda_2 t} - e^{\lambda_1 t}) p_t(x,y) dt
\]
for any $x \neq y$. For any $t > 0$ there exists $T(t) \in [\lambda_1 t,\lambda_2 t]$ such that
\[
e^{\lambda_2 t} - e^{\lambda_1 t} = (\lambda_2 t - \lambda_1 t) \frac{e^{\lambda_2 t} - e^{\lambda_1 t}}{\lambda_2 t - \lambda_1 t} = (\lambda_2 t - \lambda_1 t) e^{T(t)} \leq (\lambda_2 - \lambda_1)t e^{\Lambda t},
\]
which implies that
\[
G_{\lambda_2}(x,y) - G_{\lambda_1}(x,y) \leq (\lambda_2 - \lambda_1) \int_0^\infty t e^{\Lambda t} p_t(x,y) dt.
\]
For $\alpha = (\lambda_0(\hat{M}) - \Lambda)/2$, using that $\alpha t < e^{\alpha t}$ for any $t > 0$, we derive that
\[
G_{\lambda_2}(x,y) - G_{\lambda_1}(x,y) \leq \frac{\lambda_2 - \lambda_1}{\alpha} \int_0^\infty e^{(\Lambda + \alpha) t} p_t(x,y) dt =  \frac{\lambda_2 - \lambda_1}{\alpha} G_{\Lambda + \alpha}(x,y)
\]
for any $x \neq y$. The asserted bound follows from the compactness of $K$ and the smoothness of $G_{\Lambda + \alpha}$ away from the diagonal.
\end{proof}

\begin{prop}\label{grad Green estimate}
For any compact $K \subset \hat{M}$ and $0 < \Lambda < \lambda_0(\hat{M})$, there exists $C(K,\Lambda) > 0$ such that 
\[
\| \nabla_x G_{\lambda_2}(x,y) - \nabla_x G_{\lambda_1}(x,y)\| \leq C(K,\Lambda)(\lambda_2 - \lambda_1)
\]
for any $0 \leq \lambda_1 < \lambda_2 \leq \Lambda$ and $x, y \in K$ with $d(x,y) \geq 1/2$.
\end{prop}

\begin{proof}
After fixing $0 \leq \lambda_1 < \lambda_2 \leq \Lambda$, we consider the positive smooth function $F \colon \hat{M} \times \hat{M} \setminus \{(x,x) \colon x \in \hat{M}\} \to (0,\infty)$ defined by
\[
F(x,y) = \frac{G_{\lambda_2}(x,y) - G_{\lambda_1}(x,y)}{\lambda_2 - \lambda_1}.
\]
Since $G_\lambda(\cdot,y)$ is a $\lambda$-eigenfunction away from $y$, it can be checked easily that
\[
(\Delta_x - \lambda_1) F(\cdot,y) = G_{\lambda_2}(\cdot , y)
\]
in $\hat{M} \setminus \{y\}$. We derive from Theorem \ref{CY} applied to $(\Delta_x - \lambda_1) F(\cdot,y)$ that there exists $c_1 > 0$ depending only on $m$, $b$ and $\lambda_0(\hat{M})$ such that
\[
\| \nabla_x  (\Delta_x - \lambda_1) F(x,y) \| \leq c_1 (\Delta_x - \lambda_1) F(x,y)
\]
for any $x, y \in \hat{M}$ with $d(x,y) \geq 1/4$. This yields that
\begin{eqnarray}
\| \nabla_x \Delta_x F(x,y) \| &\leq& \| \nabla_x (\Delta_x -\lambda_1) F(x,y) \| + \lambda_1 \| \nabla_x  F(x,y) \| \nonumber \\
&\leq& c_1 \Delta_x F(x, y) + \lambda_0(\hat{M}) \| \nabla_x F(x,y) \| - c_1 F(x, y) \nonumber \\
&\leq& m c_1 \| \Hess_x F(x,y) \| + \lambda_0(\hat{M}) \| \nabla_x F(x,y) \|, \nonumber
\end{eqnarray}
for any $x,y \in \hat{M}$ with $d(x,y) \geq 1/4$, where we used that $|\Delta f| \leq m \| \Hess f \|$ for any $f \in C^\infty(\hat{M})$. Since
\[
\Delta_x F(x , y) = \lambda_1 F(x,y) + G_{\lambda_2}(x,y) > 0
\]
for any $x \neq y$, it follows from Theorem \ref{CY} applied to $F(\cdot,y)$ that there exists $c_2 > 0$ depending only on $m$, $b$ and $\lambda_0(\hat{M})$ such that
\[
\| \nabla_x F(x,y) \| \leq c_2 F(x,y)
\]
for any $x,y \in \hat{M}$ with $d(x,y) \geq 1/2$. The proof is completed by Lemma \ref{Green estimate}, which implies the existence of $c_3(K,\Lambda) > 0$ such that $F(x,y) \leq c_3(K,\Lambda)$ for any $x,y \in K$ with $d(x,y) \geq 1/2$.
\end{proof}

\begin{cor}\label{L Lip}
For any compact $K \subset \hat{M}$ and $0 < \Lambda < \lambda_0(\hat{M})$, there exists $C(K,\Lambda) > 0$ such that
\[
|L_{T}(\lambda_2,x,y) - L_{T}(\lambda_1,x,y)| \leq C(K,\Lambda) |\lambda_2 - \lambda_1|
\]
for any $T \geq \max\{1 , b/2\}$, $\lambda_1, \lambda_2 \in [0, \Lambda]$ and $x,y \in K$.
\end{cor}

\begin{proof}
Follows readily from Lemma \ref{Green estimate} and Proposition \ref{grad Green estimate}, using that $L_{T}(\lambda,x,y) = 0$ for $d(x,y) < 1/2$ and $\| \nabla_x \chi_{T,y}(x) \|_\infty , \|\Delta_x \chi_{T,y}(x)\|_\infty \leq c(m)$, according to (\ref{bound}).
\end{proof}

\section{Interior parametrix}\label{secauto}

Throughout this section let $M^m$ be a complete Riemannian manifold with bounded sectional curvature, $\hat{p} \colon \hat{M} \to M$ an infinite sheeted normal covering with deck transformation group $\Gamma$ such that
\begin{align}\label{hatmu}
    \mu = \min \{ \lambda_{\ess}(M) , \lambda_0(\hat{M}) \}  > 0,
\end{align}
 and fix $0 < \Lambda < \mu$ and $\delta > 0$.
If $M$ is non-compact, choose compact domain $K$ and $\chi_K^{\pm} \in C^\infty(M)$ according to Proposition \ref{estl2}. To unify the exposition, in the case where $M$ is compact, set $\chi_K^- = 0$.
Let $F\subseteq\hat M$ be a Dirichlet fundamental domain for $\hat{p}$, and set $\tilde{\chi}_K^- = \chi_K^- \circ \hat{p}$.

Given any $n \in \mathbb{N}$ and $n$-sheeted intermediate covering $p_\vf \colon M_\vf \to M$, consider the space $C^\infty_\vf(\hat{M},V_n^0)$ and the subspace
\[
C^{\infty}_{c,\vf}(\hat{M},V_n^0) = \{ f \in C^\infty_\vf(\hat{M},V_n^0) \colon \text{supp} f \cap F \text{ compact} \}. \nonumber
\]
It should be noticed that $C^{\infty}_{c,\vf}(\hat{M},V_n^0)$ is isomorphic to $C^\infty_c(M_\vf) \cap L^2_{\textrm{new}}(M_\vf)$.
Denote by $L^2_\vf(\hat{M},V_n^0)$ the completion of $C^{\infty}_{c,\vf}(\hat{M},V_n^0)$ with respect to
\[
\| f \|_{L^2(F)} = \int_{F} \| f(z) \|_{V_n^0}^2 dz.
\]
Then the aforementioned isomorphism extends to an isomorphism between $ L^2_\vf(\hat{M},V_n^0)$ and $L^2_{\textrm{new}}(M_\vf)$. Under these identifications, the Laplacian on $M_\vf$
\[
\Delta \colon C^\infty_c(M_\vf) \cap L^2_{\textrm{new}}(M_\vf) \to L^2_{\textrm{new}}(M_\vf)
\]
corresponds to the Laplacian (acting component-wise)
\[
\Delta \colon C^{\infty}_{c,\vf}(\hat{M},V_n^0) \subset L^2_\vf(\hat{M},V_n^0) \to L^2_\vf(\hat{M},V_n^0).
\]
Since the first operator is essentially-self adjoint, it follows that so is the latter one. Denote by $\Delta_\vf$ its closure and let $\mathcal{D}(\Delta_\vf)$ be its domain of definition, which is the completion of $C^{\infty}_{c,\vf}(\hat{M},V_n^0)$ with respect to the graph norm
\[
\| f \|_{\Delta_\vf} = \| f \|_{L^2(F)} + \| \Delta f \|_{L^2(F)}.
\]
Then $\mathcal{D}(\Delta_\vf)$ is isomorphic to $\mathcal{D}(\Delta_{M_\vf}) \cap L^2_{\textrm{new}}(M_\vf)$. Since $\Delta$ is essentially self-adjoint, we obtain that $\Delta_\vf$ coincides with the adjoint $\Delta^*$, which yields that if
\begin{equation}\label{adj}
\langle f , \Delta g \rangle_{L^2(F)} = \langle h , g \rangle_{L^2(F)}
\end{equation}
for some $f,h \in L^2_\vf(\hat{M},V_n^0)$ and any $g \in C^\infty_{c,\vf}(\hat{M}, V_n^0)$, then $f \in \mathcal{D}(\Delta_\vf)$ and $\Delta_\vf f = h$. It is noteworthy that it is actually sufficient that (\ref{adj}) holds for any evenly covered, precompact, open domain $U \subset M$ and $g \in C^\infty_{c,\vf}(\hat{M}, V_n^0)$ supported in $\hat p^{-1}(U)$, since any $g \in C^\infty_{c,\vf}(\hat{M}, V_n^0)$ can be written as a sum of such functions (e.g. consider a partition of unity on $M$ with functions supported in such domains and cut-off $g$ using the lifted functions to $\hat{M}$).

\begin{lem}\label{inner prod alt}
Let $g \in C^\infty_{c,\vf}(\hat{M},V_n^0)$ be supported in the preimage $\hat p^{-1}(U)$ of an evenly covered, precompact, open domain $U$ of $M$. If $\hat{U}$ is a connected component of $p^{-1}(U)$, then
\[
\langle f , g \rangle_{L^2(F)} = \int_{\hat{U}} \langle f ,g \rangle_{V_n^0}
\]
for any $f \in L^2_\vf(\hat{M},V_n^0)$.
\end{lem}

\begin{proof}
For any $f \in L^2_\vf(\hat{M},V_n^0)$, we have that $\langle f ,g \rangle_{V_n^0}$ is almost everywhere $\Gamma$-invariant and descends to some $h \in L^1(M)$ supported in $U$. Therefore,
\[
\langle f , g \rangle_{L^2(F)} = \int_F \langle f ,g \rangle_{V_n^0} = \int_M h = \int_U h = \int_{\hat{U}} (h \circ p) = \int_{\hat{U}} \langle f ,g \rangle_{V_n^0},
\]
where we used that $\hat p \colon \hat{U} \to U$ is an isometry.
\end{proof}

For any $0 \leq \lambda \leq \Lambda$, consider the operators $R_{T}(\lambda) \colon L^2(\hat{M}) \to \mathcal{D}(\Delta)$ and $\mathbb{L}_{T}(\lambda) \colon L^2(\hat{M}) \to L^2(\hat{M})$, as in Section \ref{secparin}.  Denote by $R_{T,n}(\lambda)$ and $\mathbb{L}_{T,n}(\lambda)$ the integral operators on $C^\infty(\hat{M},V_n)$ with kernels $R_{T}(\lambda,x,y)$ and $L_{T}(\lambda,x,y)$, respectively.

\begin{prop}\label{op on aut}
For any $0 \leq \lambda \leq \Lambda$ the following hold:
\begin{enumerate}
	\item The operator $R_{T,n}(\lambda)(1-\tilde{\chi}_K^-)$ is well-defined on $C^\infty_{c,\vf}(\hat{M},V_n^0)$ and extends to a bounded operator
	\[
	R_{T,n}(\lambda)(1-\tilde{\chi}_K^-) \colon L^2_\vf(\hat{M},V_n^0) \to \mathcal{D}(\Delta_\vf).
	\]
	\item  The operator $\mathbb{L}_{T,n}(\lambda)(1-\tilde{\chi}_K^-)$ is well-defined on $C^\infty_{c,\vf}(\hat{M},V_n^0)$ and extends to a bounded operator
	\[
	\mathbb{L}_{T,n}(\lambda)(1-\tilde{\chi}_K^-) \colon L^2_\vf(\hat{M},V_n^0) \to L^2_\vf(\hat{M},V_n^0).
	\]
	\item We have that
	\[
	(\Delta_\vf - \lambda) R_{T,n}(\lambda)(1-\tilde{\chi}_K^-) f = (1 - \tilde{\chi}_K^-) f + \mathbb{L}_{T,n}(\lambda)(1-\tilde{\chi}_K^-) f
	\]
	for any $f \in L^2_\vf(\hat{M},V_n^0)$.
\end{enumerate}
\end{prop}

\begin{proof}
Given any $f \in C^\infty_{c,\vf}(\hat{M},V_n^0)$ and any $x \in F$ we have that
		\[
R_{T,n}(\lambda)(1-\tilde{\chi}_K^-) f (x) = \int_{\hat{M}} R_{T}(\lambda,x,y) (1 - \tilde{\chi}_K^- (y)) f(y) dy.	
		\]
Keeping in mind that $R_{T}(\lambda,x,y) = 0$ if $d(x,y) \geq 2T$, and that the support of $1 - \tilde{\chi}_K^-$ may be written as the union of $\gamma D$ with $\gamma \in \Gamma$, where $D \subset F$ compact, we obtain from \cite[Lemma 5.4]{HideMagee21} (notice that the proof of \cite[Lemma 5.4]{HideMagee21} works  in our setting without any changes) that there exists a compact $D_1 \subset F$ and a finite $S \subset \Gamma$ depending on $T$ and $\chi_K^-$, such that the above integrad is supported in the compact $D_2 = \cup_{\gamma \in S} \gamma^{-1}D$ and the integral is zero unless $x \in D_1$. Consider a compactly supported, smooth $\psi \colon \hat{M} \to [0,1]$ with $\psi = 1$ in $D_2 \cup D$. Then
\[
R_{T,n}(\lambda)(1-\tilde{\chi}_K^-) f (x) = R_{T,n}(\lambda)(1-\tilde{\chi}_K^-) (\psi f) (x),
\]
which shows that this expression is well-defined for $x \in F$. Similarly, for any $\gamma \in \Gamma$ and $x \in F$ we have that
\[
R_{T,n}(\lambda)(1-\tilde{\chi}_K^-) f (\gamma x) = R_{T,n}(\lambda)(1-\tilde{\chi}_K^-) ((\psi \circ \gamma^{-1}) f) (\gamma x)
\]
which yields that  $R_{T,n}(\lambda)(1-\tilde{\chi}_K^-) f(x) $ is well-defined for any $x \in \hat{M}$. Using (\ref{Gamma equiv}) and the fact $\tilde{\chi}_K^-$ is $\Gamma$-invariant, it is straightforward to verify that $R_{T,n}(\lambda)(1-\tilde{\chi}_K^-) f$ is automorphic. Moreover, since $f(y) \in V_n^0$ for any $y \in \hat{M}$, we readily see that $R_{T,n}(\lambda)(1-\tilde{\chi}_K^-) f (x) \in V_n^0$ for any $x \in \hat{M}$.

Choose an orthonormal basis $\{e_i\}_{i=1}^{n-1}$ of $V_n^0$ and set $f_i(x) = \langle f(x) , e_i \rangle$. Then
\begin{eqnarray}
R_{T,n}(\lambda)(1-\tilde{\chi}_K^-) f (x) &=& R_{T,n}(\lambda)(1-\tilde{\chi}_K^-) (\psi f) (x) \nonumber\\
&=& \sum_{i=1}^{n-1} R_{T}(\lambda)(1-\tilde{\chi}_K^-) (\psi f_i) (x) e_i,\nonumber
\end{eqnarray}
for any $x \in F$, which together with (\ref{R norm}) gives that
\begin{eqnarray}
\| R_{T,n}(\lambda)(1-\tilde{\chi}_K^-) f \|^2_{L^2(F)} &\leq& \sum_{i=1}^{n-1} \|  R_{T}(\lambda)(1-\tilde{\chi}_K^-) (\psi f_i) \|_{L^2(\hat{M})}^2 \nonumber \\
&\leq& \sum_{i=1}^{n-1} \| R_{T}(\lambda) \|_{L^2} \| \psi f_i \|^2_{L^2(\hat{M})} \nonumber\\
&=& \| R_{T}(\lambda) \|_{L^2} \| \psi f \|^2_{L^2(\hat{M})} \nonumber\\
&\leq& C \| R_{T}(\lambda) \|_{L^2} \| f \|^2_{L^2(F)}, \nonumber
\end{eqnarray}
where we used that $f$ is automorphic, $0 \leq \tilde{\chi}_K^-, \psi \leq 1$ and $\psi$ is supported in a compact domain which intersects at most $C$ translates of $F$. This shows that the operator $R_{T,n}(\lambda)(1-\tilde{\chi}_K^-)$ is well-defined on $C^\infty_{c,\vf}(\hat{M},V_n^0)$ and extends to a bounded operator
\[
R_{T,n}(\lambda)(1-\chi_K^-) \colon L^2_\vf(\hat{M},V_n^0) \to L^2_\vf(\hat{M},V_n^0).
\]

Similarly, using that $L_{T}(\lambda,x,y) = 0$ if $d(x,y) \geq 2T$ and Proposition \ref{L estimate}, we obtain that $\mathbb{L}_{T,n}(\lambda)(1-\tilde{\chi}_K^-)$ is well-defined on $C^\infty_{c,\vf}(\hat{M},V_n^0)$ and extends to a bounded operator
\[
\mathbb{L}_{T,n}(\lambda)(1-\tilde{\chi}_K^-) \colon L^2_\vf(\hat{M},V_n^0) \to L^2_\vf(\hat{M},V_n^0).
\]

Let $U \subset M$ be an evenly covered, precompact, open domain of $M$, consider $g \in C^\infty_{c,\vf}(\hat{M},V_n^0)$ supported in $\hat p^{-1}(U)$, choose a connected component $\hat{U}$ of $\hat p^{-1}(U)$, and let $\phi \in C^\infty_c(\hat{M})$ with $\phi = 1$ in the $2T$-neighborhood of $\hat{U}$. Then $R_{T}(\lambda,x,y) = L_{T}(\lambda,x,y) = 0$ for any $x \in \hat{U}$ and $y \notin \{ \phi = 1 \}$, which implies that
\begin{eqnarray}
	(1 - \tilde{\chi}_K^-) f(x) &=& (1 - \tilde{\chi}_K^-) (\phi f)(x),\nonumber\\
	R_{T,n}(\lambda) (1 - \tilde{\chi}_K^-) f(x) &=& 	R_{T,n}(\lambda) (1 - \tilde{\chi}_K^-) (\phi f)(x)\nonumber\\
	\mathbb{L}_{T,n}(\lambda) (1 - \tilde{\chi}_K^-) f(x) &=& 	\mathbb{L}_{T,n}(\lambda) (1 - \tilde{\chi}_K^-) (\phi f)(x), \nonumber
\end{eqnarray}
 for any $f \in C^\infty_{c,\vf}(\hat{M},V_n^0)$ and $x \in \hat{U}$. Using these, Lemma \ref{inner prod alt} and Proposition \ref{R est},  we compute
{\allowdisplaybreaks
\begin{align*}
\langle (1 - \tilde{\chi}_K^-) f + \mathbb{L}_{T,n}&(\lambda)(1-\tilde{\chi}_K^-) f , g \rangle_{L^2(F)} \\
&= \int_{\hat{U}} \langle (1 - \tilde{\chi}_K^-) f + \mathbb{L}_{T,n}(\lambda)(1-\tilde{\chi}_K^-) f , g \rangle_{V_n^0}  \\
&= \int_{\hat{U}} \langle (1 - \tilde{\chi}_K^-) (\phi f) + \mathbb{L}_{T,n}(\lambda)(1-\tilde{\chi}_K^-) (\phi f) , g \rangle_{V_n^0} \\
&= \int_{\hat{U}} \langle (\Delta - \lambda) R_{T,n}(\lambda) (1-\tilde{\chi}_K^-) (\phi f) ,  g \rangle_{V_n^0} \\
&= \int_{\hat{U}} \langle R_{T,n}(\lambda) (1-\tilde{\chi}_K^-) (\phi f) , (\Delta - \lambda)  g \rangle_{V_n^0} \\
&= \int_{\hat{U}} \langle R_{T,n}(\lambda) (1-\tilde{\chi}_K^-) f , (\Delta - \lambda)  g \rangle_{V_n^0} \\
&= \langle R_{T,n}(\lambda) (1-\tilde{\chi}_K^-) f , (\Delta - \lambda)  g \rangle_{L^2(F)}.
\end{align*}
}Since $U$  and $g$ are arbitrary, we obtain that $R_{T,n}(\lambda) (1-\tilde{\chi}_K^-) f \in \mathcal{D}(\Delta_\vf)$ and
\begin{equation}\label{eq for smooth}
	(\Delta_\vf - \lambda) R_{T,n}(\lambda) (1-\tilde{\chi}_K^-) f = (1 - \tilde{\chi}_K^-) f + \mathbb{L}_{T,n}(\lambda)(1-\tilde{\chi}_K^-) f 
\end{equation}
for any $f \in C^\infty_{c,\vf}(\hat{M},V_n^0)$. Bearing in mind that $R_{T,n}(\lambda)(1-\tilde{\chi}_K^-)$, $(1-\tilde{\chi}_K^-)$ and $\mathbb{L}_{T,n}(\lambda)(1-\tilde{\chi}_K^-)$ are bounded operators in $L^2_\vf(\hat{M},V_n^0)$, we obtain that the first one can be regarded as a bounded operator
\[
R_{T,n}(\lambda)(1-\tilde{\chi}_K^-) \colon L^2_\vf(\hat{M},V_n^0) \to \mathcal{D}(\Delta_\vf).
\]
Finally, since both sides of (\ref{eq for smooth}) involve bounded operators, the equality readily extends to any $f \in L^2_\vf(\hat{M},V_n^0)$.
\end{proof}

In view of Proposition \ref{op on aut}, via the identifications $L^2_{\textrm{new}}(M_\vf) \cong L^2_\vf(\hat{M},V_n^0)$ and $\mathcal{D}(\Delta_{M_\vf}) \cap L^2_{\textrm{new}}(M_\vf) \cong \mathcal{D}(\Delta_\vf)$, multiplication with $1-\tilde{\chi}_K^-$ corresponds to multiplication with $1 - \chi_{K,\vf}^-$, the operator $R_{T,n}(\lambda)(1-\tilde{\chi}_K^-) \colon L^2_\vf(\hat{M},V_n^0) \to \mathcal{D}_\vf(\Delta,V_n^0)$ corresponds to a bounded operator
\[
\mathbb{M}_\vf^{\text{int}}(\lambda) \colon L^2_{\textrm{new}}(M_\vf) \to \mathcal{D}(\Delta_{M_\vf}) \cap L^2_{\textrm{new}}(M_\vf),
\]
and the operator $\mathbb{L}_{T,n}(\lambda)(1-\tilde{\chi}_K^-) \colon L^2_\vf(\hat{M},V_n^0) \to L^2_\vf(\hat{M},V_n^0)$ corresponds to a bounded operator
\[
\mathbb{L}_\vf^{\text{int}}(\lambda) \colon L^2_{\textrm{new}}(M_\vf) \to L^2_{\textrm{new}}(M_\vf),
\]
which satisfy
\begin{eqnarray}\label{int relation}
(\Delta_{M_\vf} - \lambda) \mathbb{M}_\vf^{\text{int}}(\lambda) = (1 - \chi_{K,\vf}^-) + \mathbb{L}_\vf^{\text{int}}(\lambda)
\end{eqnarray}
for any $0 \leq \lambda \leq \Lambda$.

In the case where $M$ is non-compact, consider $\mathbb{M}_\vf^K(\lambda)$ and $\mathbb{L}_\vf^K(\lambda)$ as in Section \ref{secparex}. If $M$ is compact, we have that $\chi_K^- = 0$, and we set $\mathbb{M}_\vf^K(\lambda) = \mathbb{L}_\vf^K (\lambda) = 0$. In any case, defining
\[
\mathbb{M}_\vf(\lambda) = \mathbb{M}_\vf^{\text{int}}(\lambda) + \mathbb{M}_\vf^K(\lambda) \colon L^2_{\textrm{new}}(M_\vf) \to \mathcal{D}(\Delta_{M_\vf}) \cap L^2_{\textrm{new}}(M_\vf),
\]
we conclude from (\ref{end relation}), (\ref{int relation}) and Proposition \ref{estl2} that
\begin{equation}\label{patched relation}
(\Delta_{M_\vf} - \lambda) \mathbb{M}_\vf(\lambda) = 1 + \mathbb{L}_\vf^{\text{int}}(\lambda) + \mathbb{L}_\vf^{K}(\lambda)
\end{equation}
and $\| \mathbb{L}_\vf^{K}(\lambda)\|_{L^2} \leq \delta$ for any $n \in \mathbb{N}$, $\vf \in \Hom(\Gamma,S_n)$ and $0 \leq \lambda \leq \Lambda$.

\section{Proof of main theorem}\label{secrandom}

In the setting of the previous section, letting $\rho_\vf^0 = \rho_n^0 \circ \vf$, we have that
\begin{eqnarray}
\mathbb{L}_{T,n}(\lambda)(1 - \tilde{\chi}_K^-)(f)(x) &=& \sum_{\gamma \in \Gamma} \int_{F} L_{T}(\lambda,x,\gamma^{-1} y) (1 - \tilde{\chi}_K^-(y)) f(\gamma^{-1} y) dy \nonumber \\
&=& \sum_{\gamma \in \Gamma} \int_{F} L_{T}(\lambda, \gamma  x, y) (1 - \tilde{\chi}_K^-(y)) \rho_\vf^0(\gamma^{-1})f(y) dy \nonumber 
\end{eqnarray}
for any $x \in F$ and any $f \in C^\infty_{c,\vf}(\hat{M},V_n^0)$, in view of (\ref{Gamma equiv}). Keeping in mind that $L_{T}(\lambda,x,y) = 0$ for $d(x,y) \geq 2T$, it follows from \cite[Lemma 5.4]{HideMagee21} that there exists a compact $D \subset F$ and a finite $S \subset \Gamma$ depending only on $T$ and $\chi_K^-$ such that the integrand is zero unless $x,y \in D$ and $\gamma \in S$. Therefore, we have that
\[
\mathbb{L}_{T,n}(\lambda)(1 - \tilde{\chi}_K^-)(f)(x) = \sum_{\gamma \in S} \int_{D} L_{T}(\lambda, \gamma  x, y) (1 - \tilde{\chi}_K^-(y)) \rho_\vf^0(\gamma^{-1})f(y) dy
\]
for any $x \in F$ and $f \in C^\infty_{c,\vf}(\hat{M},V_n^0)$.

Choose an orthonormal basis $\{e_i\}_{i=1}^{n-1}$ of $V_n^0$. Keeping in mind that $L^2_\vf(\hat{M} , V_n^0)$ consists of almost everywhere automorphic maps, we readily see that there is an isometric isomorphism $L^2_\vf(\hat{M} , V_n^0) \cong L^2(F) \otimes V_n^0$ given by
\[
f \mapsto \sum_{i=1}^{n-1} \langle f|_F , e_i \rangle \otimes e_i.
\]
Under this identification, the operator $\mathbb{L}_{T,n}(\lambda)(1 - \tilde{\chi}_K^-)$ corresponds to the operator
\[
\mathcal{L}_{\vf}^{T}(\lambda) = \sum_{\gamma \in S} a_\gamma^T(\lambda) \otimes \rho_\vf^0(\gamma^{-1}),
\]
where $a_\gamma^T(\lambda) \colon L^2(F) \to L^2(F)$ is the integral operator with kernel
\[
k_\gamma(\lambda, x,y) = L_{T}(\lambda,\gamma x,y) (1 - \tilde{\chi}_K^-(y))
\]
for any $\gamma \in \Gamma$. The fact that
\[
\int_F \int_F k_\gamma(\lambda,x,y)^2 dy dx = \int_D \int_D k_\gamma(\lambda,x,y)^2 dy dx < \infty
\]
means that $k_\gamma(\lambda,\cdot,\cdot) \in L^2(F \times F)$, and therefore, the operator $a_\gamma^T(\lambda)$ is Hilbert-Schmidt for any $0 \leq \lambda \leq \Lambda$ and $\gamma \in \Gamma$.

\begin{lem}\label{a Lip}
There exists $c > 0$ depending on $\Lambda$, $T$ and $\chi_K^-$ such that
\[
\| a_\gamma^T(\lambda_2) - a_\gamma^T(\lambda_1) \|_{L^2(F)} \leq c (\lambda_2 - \lambda_1)
\]
for any $0 \leq \lambda_1 \leq \lambda_2 \leq \Lambda$ and $\gamma \in S$.
\end{lem}

\begin{proof}
For any $\gamma \in S$ and $0 \leq \lambda_1 \leq \lambda_2 \leq \Lambda$ we have that
\[
| ((a_\gamma^T(\lambda_2) - a_\gamma^T(\lambda_1) ) f ) (x) | \leq \| k_\gamma(\lambda_2,x , \cdot) - k_\gamma(\lambda_1,x , \cdot) \|_{L^2(F)} \| f \|_{L^2(F)}
\]
for almost every $x \in F$ and any $f \in L^2(F)$, and hence,
\[
\|  a_\gamma^T(\lambda_2) - a_\gamma^T(\lambda_1)  \|_{L^2(F)} \leq \| k_\gamma(\lambda_2 , \cdot , \cdot) -  k_\gamma(\lambda_1 , \cdot , \cdot)  \|_{L^2(F \times F)}.
\]
Bearing in mind that $k_\gamma(\lambda_2 , \cdot , \cdot) -  k_\gamma(\lambda_1 , \cdot , \cdot)$ is supported in $D \times D$, we obtain that
\begin{eqnarray}
\|  a_\gamma^T(\lambda_2) - a_\gamma^T(\lambda_1)  \|_{L^2(F)} &\leq& |D| \| k_\gamma(\lambda_2 , \cdot , \cdot) -  k_\gamma(\lambda_1 , \cdot , \cdot)  \|_{L^\infty(D \times D)} \nonumber \\
&\leq& |D| \sup_{x,y \in D} | L_{T}(\lambda_2 , \gamma x , y) -  L_{T}(\lambda_1 , \gamma x , y) |. \nonumber
\end{eqnarray}
Since $S$ is finite and $D$ is compact, we readily see that there exists a compact $D^\prime \subset \hat{M}$ which contains the translates $\gamma D$ for all $\gamma \in S$ and the assertion follows from Corollary \ref{L Lip}.
\end{proof}

Denoting by $\rho_\infty \colon \Gamma \to \text{End}(\ell^2(\Gamma))$ the right regular representation of $\Gamma$, consider the operator $\mathcal{L}_{\infty}^{T}(\lambda) \colon L^2(F) \otimes \ell^2(\Gamma) \to  L^2(F) \otimes \ell^2(\Gamma)$ defined by
\[
\mathcal{L}_{\infty}^{T}(\lambda) = \sum_{\gamma \in S} a_{\gamma}^{T}(\lambda) \otimes \rho_{\infty}(\gamma^{-1}).
\]

\begin{lem}\label{cal L estimate}
For any $\delta > 0$, there exists $T \geq \max\{1,b/2\}$ depending on $\Lambda$ and $\delta$ such that
\[
	\| \mathcal{L}^{T}_{\infty}(\lambda) \|_{L^2(F) \otimes \ell^2(\Gamma)} \leq \delta
\]
for any $0 \leq \lambda \leq \Lambda$.
\end{lem}

\begin{proof}
It should be noticed that the map $U \colon L^2(F) \otimes \ell^2(\Gamma) \to L^2(\hat{M})$ defined by $U(f \otimes \delta_g) = (f \circ g^{-1}) \chi_{gF}$ is an isometric isomorphism. For any $f \in L^2(F)$ and $g \in \Gamma$, we compute
\begin{eqnarray}
U (\mathcal{L}_{\infty}^{T}(\lambda) (f \otimes \delta_g))(x) &=& \sum_{\gamma \in S} U (a_{\gamma}^{T}(\lambda)(f) \otimes \delta_{g\gamma})(x) \nonumber \\
&=& \sum_{\gamma \in S} \chi_{g \gamma F}(x) a_{\gamma}^{T}(\lambda)(f)(\gamma^{-1} g^{-1} x) \nonumber\\
&=& \sum_{\gamma \in S} \chi_{g \gamma F}(x) \int_{F} L_{T}(\lambda,g^{-1} x,y) (1 - \tilde{\chi}_K^-(y)) f(y) dy, \nonumber
\end{eqnarray}
and
\begin{eqnarray}
	(\mathbb{L}_{T}(\lambda)(1- \tilde{\chi}_K^-) U(f \otimes \delta_g))(x) &=& (\mathbb{L}_{T}(\lambda)(1- \tilde{\chi}_K^-) (f \circ g^{-1}) \chi_{gF})(x) \nonumber\\
	&=& \int_{gF} L_{T}(\lambda,x,y) (1 -\tilde{\chi}_K^-(y))f(g^{-1}y)dy \nonumber \\
	&=& \int_{F} L_{T}(\lambda,g^{-1}x,y) (1 -\tilde{\chi}_K^-(y))f(y)dy\nonumber
\end{eqnarray}
for almost any $x \in \hat{M}$, where we used (\ref{Gamma equiv}). From the definition of $S$, it is evident that both expressions vanish for almost any $x$ that is not in the union of $g \gamma F$ with $\gamma \in S$. Almost any $x$ in this union belongs to exactly one translate of $F$, and hence,
\[
	U (\mathcal{L}_{\infty}^{T}(\lambda) (f \otimes \delta_g))(x) = (\mathbb{L}_{T}(\lambda)(1- \tilde{\chi}_K^-) U(f \otimes \delta_g))(x)
	\]
Therefore, under the identification $L^2(F) \otimes \ell^2(\Gamma) \cong L^2(\hat{M})$ via $U$, the operator $\mathcal{L}_{\infty}^{T}(\lambda)$ corresponds to the operator $\mathbb{L}_{T}(\lambda)(1 - \tilde{\chi}_K^-)$ and the proof is completed by Proposition \ref{L estimate}.
\end{proof}

\begin{prop}\label{bound for fixed}
For any $0<\Lambda <\mu$, $\delta > 0$ and $0<\varepsilon<1$, there exists $T \geq \max\{1,b/2\}$ such that for any $\lambda \in [0,\Lambda]$ there exists a finitely supported map $a \colon\Gamma\to\Mat_{r\times r}(\C)$ (for some--large--$r\in\N$), such that 
\[
\| \mathcal{L}_\vf^T(\lambda) \| \leq \ve + 3 \delta
\]
for any $n \in \mathbb{N}$ and any $\vf\in\Hom(\Gamma,S_n)$ that satisfies \eqref{almost} for $a$ and the given $\ve$.
\end{prop}

\begin{proof}
	Fix $0<\Lambda <\mu$, $\delta > 0$. We know from Lemma \ref{cal L estimate} that there exists $T \geq \max\{1,b/2\}$ such that
	\[
	\| \mathcal{L}^{T}_{\infty}(\lambda) \|_{L^2(F) \otimes \ell^2(\Gamma)} \leq \delta
	\]
	for any fixed $\lambda \in[0, \Lambda]$. Since $a_\gamma^T(\lambda) \colon L^2(F) \to L^2(F)$ is Hilbert-Schmidt, and thus, compact for any $\gamma \in \Gamma$, it can be approximated by operators with finite-dimensional range. Keeping in mind that $S$ is finite, we derive that there exists a finite-dimensional subspace $V \subset L^2(F)$ and operators $b_\gamma(\lambda) \colon V \to V$ such that
	\[
	\| a_\gamma^{T}(\lambda) - b_\gamma(\lambda) \|_{L^2(F)} \leq \frac{\delta}{|S|}
	\]
	for any $\gamma \in S$. Identifying $V$ with $\mathbb{R}^r \subset \mathbb{C}^r$ for some $r \in \mathbb{N}$ and regarding $b_\gamma(\lambda) \in \Mat_{r \times r}(\mathbb{C})$ for any $\gamma \in S$, consider the map $a \colon \Gamma \to \Mat_{r \times r}(\mathbb{C})$ supported in $S$ with $a(\gamma) = b_\gamma(\lambda)$ for any $\gamma \in S$. Consider now $\vf \in \Hom(\Gamma,S_n)$ that satisfies \eqref{almost} for this map $a$ and the given $\ve$, that is,
	\[
	\| \sum_{\gamma \in S} b_\gamma(\lambda) \otimes \rho_\vf^0(\gamma^{-1}) \|_{V \otimes V_n^0} \leq \| \sum_{\gamma \in S} b_\gamma(\lambda) \otimes \rho_\infty(\gamma^{-1}) \|_{V \otimes \ell^2(\Gamma)} + \ve.
	\]
	Combining the above, we conclude that
	\begin{eqnarray}
		\| \mathcal{L}_{\vf}^{T}(\lambda) \|_{L^2(F) \otimes V_n^0} &\leq& \|  \sum_{\gamma \in S} b_\gamma(\lambda) \otimes \rho_\vf^0(\gamma^{-1}) \|_{L^2(F) \otimes V_n^0} + \delta \nonumber\\
		&\leq&	\|\sum_{\gamma \in S} b_\gamma(\lambda) \otimes \rho_\infty(\gamma^{-1}) \|_{L^2(F) \otimes \ell^2(\Gamma)} + \varepsilon + \delta \nonumber\\
		&\leq &  \| \mathcal{L}_{\infty}^{T}(\lambda) \|_{L^2(F) \otimes \ell^2(\Gamma)} + \ve + 2 \delta \leq \ve + 3 \delta, \nonumber
	\end{eqnarray}
as we wished.
\end{proof}

\begin{cor}\label{bound for all}
	For any $0<\Lambda <\mu$, $\delta > 0$ and $0<\varepsilon<1$, there are $T \geq \max\{1,b/2\}$, integers $r_i\ge1$ and finitely supported maps $a_i\colon\Gamma\to\Mat_{r_i\times r_i}(\C)$, $1\le i\le k$ (for some--large--$k\in\N$), such that 
	\[
	\| \mathcal{L}_\vf^T(\lambda) \| \leq \ve + 4 \delta
	\]
	for any $\lambda \in [0,\Lambda]$, $n \in \mathbb{N}$ and any $\vf\in\Hom(\Gamma,S_n)$ that satisfies \eqref{almost} with input $r=r_i$, $a=a_i$, and the given $\ve$, for each $1\le i\le k$.
\end{cor}

\begin{proof}
	Choose $T \geq \max\{1,b/2\}$ according to Proposition \ref{bound for fixed}. We know from Lemma \ref{a Lip} that there exists $c > 0$ depending on $\Lambda$, $T$ and $\chi_K^-$ such that
	\[
	\| a_\gamma^T(\lambda_1) - a_\gamma^T(\lambda_2) \|_{L^2(F)} \leq c |\lambda_1 - \lambda_2|
	\]
	for any $\lambda_1,\lambda_2 \in [0,\Lambda]$, which yields that
	\begin{equation}\label{cal L Lip}
		\| \mathcal{L}_{\vf}^{T}(\lambda_1) -  \mathcal{L}_{\vf}^{T}(\lambda_2)  \|_{L^2(F) \otimes V_n^0} \leq c |S| |\lambda_1 - \lambda_2|
	\end{equation}
	for any $\lambda_1,\lambda_2 \in [0,\Lambda]$, $n \in \mathbb{N}$ and $\vf \in \Hom(\Gamma,S_n)$. Fix a finite subset $X$ of $[0,\Lambda]$ such that $d(\lambda,X) < \delta/(c|S|)$ any $\lambda \in [0,\Lambda]$. Since $X$ is finite, we obtain from Proposition \ref{bound for fixed} that there exist finitely many, finitely supported maps $a_i\colon\Gamma\to\Mat_{r_i\times r_i}(\C)$, such that
	\[
	\| \mathcal{L}_{\vf}^{T}(\lambda) \|_{L^2(F) \otimes V_n^0} \leq \ve + 3 \delta
	\]
	for any $\lambda \in X$, $n \in \mathbb{N}$ and $\varphi \in \Hom(\Gamma,S_n)$ satisfying \eqref{almost} for any of the $a_i$ and the given $\ve$. The assertion follows from the choice of $X$ together with (\ref{cal L Lip}).
\end{proof}

\begin{thm}\label{main}
	Assume that the Ricci curvature of $M$ is bounded from below
	and that $\mu = \inf\{\lambda_0(\hat M),\lambda_{\ess}(M)\} > 0$.
	Fix $0<\Lambda<\mu$ and $0<\ve<1$. Then there are integers $r_i\ge1$ and finitely supported maps $a_i\colon\Gamma\to\Mat_{r_i\times r_i}(\C)$, $1\le i\le k$ (for some--large--$k\in\N$), such that $M_\vf$ is $[0,\Lambda]$-stable if $\vf\in\Hom(\Gamma,S_n)$ satisfies \eqref{almost} with input $r=r_i$, $a=a_i$, and the given $\ve$, for each $1\le i\le k$.
\end{thm}

\begin{proof}
If $M$ is non-compact, let $K$ and $\chi_{K}^{\pm}$ according to Proposition \ref{estl2} corresponding to $\Lambda$ and $\delta = (1-\ve)/10$. In the case where $M$ is compact, set $\chi^-_K = 0$. Let $-(m-1)b^2$ be a lower bound for the Ricci curvature of $M$, where $m$ is the dimension of $M$, and consider $T \geq \max\{1,b/2\}$ and finitely many maps $a_i$ according to Corollary \ref{bound for all}. Suppose that $\vf \in \Hom(\Gamma,S_n)$ satisfies \eqref{almost} for any $a_i$ and the given $\ve$. We know from (\ref{patched relation}) that
\[
(\Delta_{M_\vf} - \lambda) \mathbb{M}_\vf(\lambda) = 1 + \mathbb{L}_\vf^{\text{int}}(\lambda) + \mathbb{L}_\vf^{K}(\lambda)
\]
in $L^2_{\text{new}}(M_\vf)$ and $\| \mathbb{L}_\vf^{K}(\lambda)\|_{L^2} \leq \delta$ for any $\lambda \in [0,\Lambda]$. Bearing in mind that under the identification $L^2_{\text{new}}(M_\vf) \cong L^2_\rho(\hat{M},V_n^0) \cong L^2(F) \otimes V_n^0$, the operator $\mathbb{L}_\vf^{\text{int}}(\lambda)$ corresponds to $\mathcal{L}_{\vf}^{T}(\lambda)$, we derive from Corollary \ref{bound for all} that $\| \mathbb{L}_\vf^{\text{int}}(\lambda) \| \leq \ve + 4 \delta$ for any $\lambda \in [0,\Lambda]$. Therefore, the operator $\mathbb{T}_\vf(\lambda) = \mathbb{L}_\vf^{\text{int}}(\lambda) + \mathbb{L}_\vf^{K}(\lambda)$ has norm $\|\mathbb{T}_\vf(\lambda)\| \leq \ve + 5 \delta<1$, and thus, $1 + \mathbb{T}_\vf(\lambda)$ is invertible with inverse $\sum_{n=0}^\infty (-\mathbb{T}_\vf(\lambda))^n$ for any $\lambda \in [0,\Lambda]$. We deduce that
\[
(\Delta_{M_\vf} - \lambda) \mathbb{M}_\vf(\lambda) (1 +\mathbb{T}_\vf(\lambda))^{-1} = 1
\]
in $L^2_{\text{new}}(M_\vf)$, which implies that the operator
\[
\Delta_{M_\vf} - \lambda \colon \mathcal{D}(\Delta_{M_\vf}) \cap L^2_{\text{new}}(M_\vf) \to L^2_{\text{new}}(M_\vf)
\]
is surjective, and hence, injective for any $\lambda \in [0,\Lambda]$, this operator being self-adjoint. We conclude that $\Delta_{M_\vf}$ does not have new eigenvalues in $[0,\Lambda]$.
\end{proof}

Theorem \ref{main} readily yields the more general versions of Theorems \ref{mainth}, \ref{mainth2} and Corollary \ref{mainco} for intermediate coverings.


\end{document}